\documentclass[11pt]{article}
\usepackage{amsmath,amsfonts,amsthm,amssymb, mathtools}
\usepackage{mathrsfs, graphicx,color,latexsym, tikz, calc,
}
\usepackage{tikz-cd}
\usepackage{graphicx}
\usepackage{epstopdf}
\usepackage{enumerate}
\usepackage{caption}
\usepackage{stmaryrd}
\usepackage{xcolor}
\usepackage{hyperref}
\hypersetup{
    colorlinks=true,
    citecolor=blue,
    linkcolor=blue,
    urlcolor=blue
}
\usepackage{cite}

\usetikzlibrary{shadows}
\usetikzlibrary{patterns,arrows,decorations.pathreplacing}

\usepackage[margin=2.5cm]{geometry}
\usepackage{ulem}

\newtheorem{theorem}{Theorem}[section]

\newtheorem{lemma}[theorem]{Lemma}

\newtheorem{problem}[theorem]{Problem}
\newtheorem{conjecture}[theorem]{Conjecture}

\allowdisplaybreaks[4]

\title{\bf \Large }
\date{\today }
\title{{\bf \Large 
The binomial norm  of intersecting-union families}\footnote{ Lihua Feng was supported by the NSFC (Nos. 12271527 and 12471022) and NSF of Qinghai Province (No. 2025-ZJ-902T). E-mail addresses: \url{wuyjmath@163.com} (Y. Wu), \url{fenglh@163.com} (L. Feng). }
\author{
{\small  Yongjiang Wu,\ \ Lihua Feng\footnote{Corresponding author}
}\\[2mm]
\small School of Mathematics and Statistics, HNP-LAMA, Central South University\\
 \small Changsha, Hunan, 410083, China\\ 
}}

\begin{document}
\maketitle
\begin{abstract}
In a 2021 survey on Katona's circle method, Frankl conjectured that every  family $\mathcal{F}\subseteq 2^{[n]}$ in which any two members intersect and no two members cover $[n]$ satisfies the sharp  binomial norm bound
$
  \lVert\mathcal F\rVert_n
  :=\sum_{F\in\mathcal F}\binom{n}{|F|}^{-1}
  \leq \frac{n+1}{6}.
$
This improves the earlier estimate $\frac{n}{4}$ obtained by the circle method. In this paper, we prove  Frankl's conjecture and  determine all extremal families.
Our proof develops a continuous $p$-biased measure approach in place of the circle method. The intersection and union conditions lead to a sharp estimate for
$
\mu_p(\mathcal F)+\mu_{1-p}(\mathcal F).
$
Integrating this estimate over $p$ converts it directly into the desired binomial norm bound and recovers the optimal coefficient $\frac{1}{6}$. This continuous averaging is the key new ingredient of the proof and also yields the characterization of all extremal families.
\end{abstract}



\section{Introduction}
Let $[n]=\{1,2,\ldots,n\}$ and let $2^{[n]}$ denote its power set. For a family $\mathcal{F}\subseteq 2^{[n]}$, write $\mathcal{F}^{(k)}=\{F\in\mathcal{F}:|F|=k\}$. A family is called \textit{intersecting} if any two of its members have non-empty intersection. It is called \textit{union} if the union of any two of its members is not the whole ground set. In this paper we consider families satisfying these two dual  restrictions simultaneously.

A family $\mathcal{F}\subseteq 2^{[n]}$ is called an \textit{intersecting-union family}, abbreviated \textit{IU-family}, if for all $F,G\in\mathcal{F}$,
$$
F\cap G\neq\emptyset \qquad\text{and}\qquad F\cup G\neq [n].
$$
The study of IU-families dates back to the 1970s. Daykin and Lov\'asz \cite{DL} proved that every IU-family satisfies the simple size bound
$$
|\mathcal{F}|\le 2^{n-2},
$$
a result also obtained independently by Seymour \cite{Seymour}. This bound is best possible and can be proved elegantly using the Harris--Kleitman correlation inequality \cite{Harris,Kleitman}.
More recent work has extended this line of investigation in two directions: Frankl and Kupavskii \cite{FK2025} proved a generalization of this result in the setting of integer sequences, whereas Frankl and Wang \cite{FW26} further studied the extremal cardinality problem for families with more general intersection and union constraints. 
 While the size bound gives a complete answer to the extremal cardinality question, it treats all sets equally regardless of their sizes and thus does not capture the finer distributional structure of the family. A more refined size parameter records the proportion of each layer of
the Boolean lattice occupied by a family. Following Frankl
\cite{FranklBinomialNorm}, for $\mathcal F\subseteq 2^{[n]}$, we define
its \textit{binomial norm} by
$$
  \lVert\mathcal F\rVert_n
  :=\sum_{F\in\mathcal F}\binom{n}{|F|}^{-1}.
$$
Equivalently,
$$
  \lVert\mathcal F\rVert_n
  =\sum_{k=0}^{n}
    \frac{|\mathcal F^{(k)}|}{\binom{n}{k}},
$$
Thus, the binomial norm is the sum of the densities of $\mathcal F$ in
the individual layers of the Boolean lattice.

Recall that a family $\mathcal F\subseteq 2^{[n]}$ is an \textit{antichain} if
no two distinct members $F,G\in\mathcal F$ satisfy $F\subseteq G$.
The classical LYM inequality asserts that every antichain  $\mathcal F\subseteq 2^{[n]}$ satisfies
$$
  \lVert\mathcal F\rVert_n\leq 1.
$$
Historically, Yamamoto \cite{Yamamoto} proved this inequality in 1954, Meshalkin \cite{Meshalkin}  gave
an alternative proof and generalization in 1963,  Bollob\'as's more general
set-pairs inequality \cite{Bollobas1965} appeared in 1965, and Lubell gave a short proof
in 1966 \cite{Lubell}.
Accordingly, we use the name \textit{LYM inequality} for the classical
antichain result and the term \textit{binomial norm} for the function
$\lVert\cdot\rVert_n$.

The binomial norm also has a natural extremal theory beyond antichains.
Frankl \cite{FranklBinomialNorm} proved, in particular, that if a family
$\mathcal F\subseteq 2^{[n]}$ contains no $s$ pairwise disjoint members,
then
$$
  \lVert\mathcal F\rVert_n
  \leq \frac{s-1}{s}(n+1),
$$
and established a corresponding inequality for cross-dependent
families, together with binomial-norm versions of the theorems of
Harper and Katona. For $s=2$, this gives
$$
  \lVert\mathcal F\rVert_n\leq \frac{n+1}{2}
$$
for every intersecting family $\mathcal F$. For IU-families, which are
generally not antichains, the analogous problem is to determine the
largest possible binomial norm under the two dual restrictions.

In a 2021 survey on old and new applications of Katona's circle method, Frankl \cite{Frankl2021} addressed this question. Using the circle method, a powerful averaging technique introduced by Katona \cite{Katona1974} that has found numerous applications in extremal set theory, he obtained the bound
$$
 \lVert\mathcal F\rVert_n\leq \frac{n}{4}
$$
for every IU-family $\mathcal{F}$. He further proposed the following conjecture.

\begin{conjecture}[Frankl, Conjecture 11.4 in \cite{Frankl2021}]
For every IU-family $\mathcal{F}\subseteq 2^{[n]}$,
$$
\lVert\mathcal F\rVert_n\leq  \frac{n+1}{6}.
$$
\end{conjecture}

The purpose of this paper is to prove Frankl's conjecture and to determine its extremal structure.
Our main result is the following.

\begin{theorem}\label{tm1}
Let $\mathcal{F}\subseteq 2^{[n]}$ be an IU-family. Then
$$
\lVert\mathcal F\rVert_n\leq  \frac{n+1}{6}.
$$
Moreover, equality holds if and only if $\mathcal{F}=\{F\subseteq [n]: i\in F,\ j\notin F\}$ for some distinct $i,j\in [n]$.
\end{theorem}

The proof is short and uses a different probabilistic viewpoint. Instead of cyclic permutations and discrete averaging, we work with  $p$-biased measures on the Boolean lattice. The family $\mathcal{F}$ is embedded into the intersection of its upset and downset. The IU conditions imply that the upset is intersecting and the complement of the downset is also intersecting. Applying the $p$-biased Erd\H{o}s--Ko--Rado theorem to these two families and then using the Harris--Kleitman correlation inequality yields a sharp pointwise estimate for $\mu_p(\mathcal{F})+\mu_{1-p}(\mathcal{F})$. Integrating this estimate over $0\le p\le \frac{1}{2}$ via the Beta-integral identity extracts the exact coefficient $\frac{1}{6}$. This integration step has no analogue in the discrete averaging of the circle method and is exactly where the improvement from $\frac{n}{4}$ to $\frac{n+1}{6}$ comes from.

A natural generalization arises by replacing the two dual restrictions with quantitative versions. Given integers $t,u\ge 1$, a family $\mathcal{F}\subseteq 2^{[n]}$ is called a \textit{$(t,u)$-family} if for all $F,G\in\mathcal{F}$,
$$
|F\cap G|\ge t,\qquad |F\cup G|\le  n-u.
$$
The following extremal problem then arises naturally.

\begin{problem}
Determine
$$
  \max\bigl\{
    \lVert\mathcal F\rVert_n:
    \mathcal F\subseteq 2^{[n]}
    \text{ is a $(t,u)$-family}
  \bigr\}.
$$
\end{problem}

A plausible extremal construction  is
$
\mathcal{E}_{T,U}=\{A\subseteq [n]: T\subseteq A,\ A\cap U=\emptyset\},
$
where $T,U\subseteq[n]$ are disjoint with $|T|=t$ and $|U|=u$. The binomial norm of $\mathcal{E}_{T,U}$ is
$$
\lVert\mathcal E_{T,U}\rVert_n
  =\sum_{F\in\mathcal E_{T,U}}\binom{n}{|F|}^{-1}
  =\sum_{r=t}^{n-u}
    \frac{\binom{n-t-u}{r-t}}{\binom{n}{r}}.
$$
When $t=u=1$,  $\mathcal{E}_{T,U}$ reduces to the extremal family of Theorem \ref{tm1}.

The remainder of this paper is organized as follows. Section \ref{se2} introduces the necessary tools, including the $p$-biased Erd\H{o}s--Ko--Rado theorem, the Harris--Kleitman correlation inequality, and the Beta-integral identity. Section \ref{se3} presents the proof of Theorem \ref{tm1}.
% Section \ref{se4} contains some concluding remarks and open problems.

\section{Tools}\label{se2}

We shall work with the \textit{$p$-biased measure} on the Boolean lattice. For $0\le p\le 1$ and a family $\mathcal{F}\subseteq 2^{[n]}$, define
$$
\mu_p(\mathcal{F})=\sum_{F\in\mathcal{F}}p^{|F|}(1-p)^{n-|F|}.
$$
Equivalently, $\mu_p(\mathcal{F})$ is the probability that a random subset of $[n]$, obtained by including each element independently with probability $p$, belongs to $\mathcal{F}$.
A family $\mathcal{F}\subseteq 2^{[n]}$ is called \textit{increasing} if $F\in\mathcal{F}$ and $F\subseteq G\subseteq [n]$ imply $G\in\mathcal{F}$. It is called \textit{decreasing} if $F\in\mathcal{F}$ and $G\subseteq F$ imply $G\in\mathcal{F}$. For a family $\mathcal{F}$, define its complement by $\overline{\mathcal{F}}=\{[n]\setminus F:F\in\mathcal{F}\}$.

We need two standard facts. The first is the $p$-biased Erd\H{o}s--Ko--Rado theorem, which bounds the measure of an intersecting family under the condition $p\le \frac{1}{2}$.

\begin{theorem}[\cite{fishburn1986,Friedgut2008}]\label{le1}
If $\mathcal{F}\subseteq 2^{[n]}$ is intersecting and $0\le p\le \frac{1}{2}$, then
$$
\mu_p(\mathcal{F})\le p.
$$
Moreover, if $0<p<\frac{1}{2}$, equality holds if and only if $\mathcal{F}=\{F\subseteq [n]:i\in F\}$ for some $i\in[n]$.
\end{theorem}

The second is an extension of the Harris--Kleitman correlation inequality, which states that increasing and decreasing families are negatively correlated.

\begin{lemma}[\cite{Siggers2012, FranklTokushige2018}]\label{le2}
If $\mathcal{F}\subseteq 2^{[n]}$ is increasing and $\mathcal{G}\subseteq 2^{[n]}$ is decreasing, then
$$
\mu_p(\mathcal{F}\cap\mathcal{G})\le \mu_p(\mathcal{F})\mu_p(\mathcal{G}).
$$
\end{lemma}

We  also need the following elementary identity.

\begin{lemma}[Beta-integral identity]\label{le3}
For $0\le k\le n$,
$$
\int_0^1 p^k(1-p)^{n-k}\,dp=\frac{k!(n-k)!}{(n+1)!}=\frac{1}{(n+1)\binom{n}{k}}.
$$
\end{lemma}

\begin{proof}
Let $I(a,b)=\int_0^1 p^a(1-p)^b\,dp$, where $a,b\ge 0$. We first derive a recurrence relation using integration by parts. In the form
$
\int_0^1 u\,dv=\left.uv\right|_0^1-\int_0^1 v\,du,
$
set
$$
u=(1-p)^b,\qquad dv=p^a\,dp.
$$
Then
$$
du=-b(1-p)^{b-1}\,dp,\qquad v=\frac{p^{a+1}}{a+1},
$$
where we omit the constant of integration since it cancels in the definite integral. Hence,
$$
I(a,b)=\left.\frac{p^{a+1}}{a+1}(1-p)^b\right|_0^1+\frac{b}{a+1}\int_0^1 p^{a+1}(1-p)^{b-1}\,dp.
$$
At $p=1$, we have $(1-p)^b=0$ (for $b\ge 1$), and at $p=0$, we have $p^{a+1}=0$. Thus,
$$
I(a,b)=\frac{b}{a+1}I(a+1,b-1)\qquad(b\ge 1). 
$$

Now apply this recurrence repeatedly with $a=k$ and $b=n-k$. After $n-k$ steps, we obtain
$$
I(k,n-k)=\frac{(n-k)(n-k-1)\cdots 1}{(k+1)(k+2)\cdots n}\,I(n,0).
$$
Since $I(n,0)=\int_0^1 p^n\,dp=\frac{1}{n+1}$, we obtain
$$
I(k,n-k)=\frac{(n-k)!}{(k+1)(k+2)\cdots n}\cdot \frac{1}{n+1}=\frac{k!(n-k)!}{n!(n+1)}=\frac{k!(n-k)!}{(n+1)!}=\frac{1}{(n+1)\binom{n}{k}}.
$$
This proves the lemma.
\end{proof}

\section{Proof of Theorem \ref{tm1}}\label{se3}

Let $\mathcal{F}\subseteq 2^{[n]}$ be an IU-family. Define its \textit{upset} and \textit{downset} by
$$
\mathcal{F}^{\uparrow}=\{U\subseteq [n]: F\subseteq U\text{ for some }F\in\mathcal{F}\},\qquad
\mathcal{F}^{\downarrow}=\{D\subseteq [n]: D\subseteq F\text{ for some }F\in\mathcal{F}\}.
$$
Then $\mathcal{F}^{\uparrow}$ is increasing, $\mathcal{F}^{\downarrow}$ is decreasing, and $\mathcal{F}\subseteq \mathcal{F}^{\uparrow}\cap\mathcal{F}^{\downarrow}$.

The IU conditions imply two structural facts. First, $\mathcal{F}^{\uparrow}$ is intersecting. Indeed, if $U,V\in\mathcal{F}^{\uparrow}$, then $U\supseteq F$ and $V\supseteq G$ for some $F,G\in\mathcal{F}$, and therefore $U\cap V\supseteq F\cap G\neq\emptyset$. Second, no two members of $\mathcal{F}^{\downarrow}$ cover $[n]$. Indeed, if $U,V\in\mathcal{F}^{\downarrow}$, then $U\subseteq F$ and $V\subseteq G$ for some $F,G\in\mathcal{F}$, which implies $U\cup V\subseteq F\cup G\neq[n]$. Equivalently, $\overline{\mathcal{F}^{\downarrow}}$ is intersecting.

Fix $0\le p\le \frac{1}{2}$ and put $q=1-p$. Set
$$
a=\mu_p(\mathcal{F}^{\uparrow}),\quad b=\mu_q(\mathcal{F}^{\uparrow}),\quad c=\mu_p(\mathcal{F}^{\downarrow}),\quad d=\mu_q(\mathcal{F}^{\downarrow}).
$$
Since $\mathcal{F}^{\uparrow}$ is intersecting, it cannot contain both a set and its complement. As $\mu_q(\mathcal{F}^{\uparrow})=\mu_p(\overline{\mathcal{F}^{\uparrow}})$, we have
\begin{equation}\label{f1}
a+b\le 1. 
\end{equation}
Similarly, the union property of $\mathcal{F}^{\downarrow}$ implies that it cannot contain both a set and its complement. Hence,
\begin{equation}\label{f2}
c+d\le 1. 
\end{equation}

Applying Theorem \ref{le1} to the intersecting family $\mathcal{F}^{\uparrow}$ yields
\begin{equation}\label{f3}
a\le p. 
\end{equation}
Moreover, since $\overline{\mathcal{F}^{\downarrow}}$ is intersecting, the same theorem gives
\begin{equation}\label{f4}
d=\mu_q(\mathcal{F}^{\downarrow})=\mu_p(\overline{\mathcal{F}^{\downarrow}})\le p. 
\end{equation}
Since $\mathcal{F}\subseteq \mathcal{F}^{\uparrow}\cap\mathcal{F}^{\downarrow}$, and $\mathcal{F}^{\uparrow}$ is increasing and $\mathcal{F}^{\downarrow}$ is decreasing, Lemma \ref{le2} gives
$$
\mu_p(\mathcal{F})\le \mu_p(\mathcal{F}^{\uparrow}\cap\mathcal{F}^{\downarrow})\le ac.
$$
Similarly, we have
$$
\mu_q(\mathcal{F})\le \mu_q(\mathcal{F}^{\uparrow}\cap\mathcal{F}^{\downarrow})\le bd.
$$
Combining these with \eqref{f1} and \eqref{f2}, we get
$$
\mu_p(\mathcal{F})+\mu_q(\mathcal{F})
\le ac+bd
\le a(1-d)+(1-a)d
= a+d-2ad.
$$
By \eqref{f3} and \eqref{f4}, we  have $0\le a,d\le p\le \frac{1}{2}$. Then the function $a+d-2ad$ is increasing in both variables. Thus,
\begin{equation}\label{f5}
\mu_p(\mathcal{F})+\mu_{1-p}(\mathcal{F})\le 2p(1-p),\qquad 0\le p\le \frac12. 
\end{equation}

It remains to recover the binomial norm from this measure estimate. Integrating \eqref{f5} over $0\le p\le \frac{1}{2}$, we obtain
$$
\int_0^1 \mu_p(\mathcal{F})\,dp
= \int_0^{\frac{1}{2}}\bigl(\mu_p(\mathcal{F})+\mu_{1-p}(\mathcal{F})\bigr)\,dp
\le \int_0^{\frac{1}{2}}2p(1-p)\,dp
= \frac16.
$$
By Lemma \ref{le3},
$$
\begin{aligned}
  \int_0^1\mu_p(\mathcal F)\,dp
  &=\sum_{F\in\mathcal F}
    \int_0^1p^{|F|}(1-p)^{n-|F|}\,dp=\frac{1}{n+1}
    \sum_{F\in\mathcal F}\binom{n}{|F|}^{-1}\\
  &=\frac{1}{n+1}\lVert\mathcal F\rVert_n.
\end{aligned}
$$
Consequently,
$$
  \lVert\mathcal F\rVert_n\leq\frac{n+1}{6}.
$$

We now characterize the equality cases. If equality holds in Theorem \ref{tm1}, then equality must hold throughout the chain of inequalities in the proof. Define 
$$
\Delta(p)=2p(1-p)-\bigl(\mu_p(\mathcal{F})+\mu_{1-p}(\mathcal{F})\bigr).
$$
From  \eqref{f5}, $\Delta(p)\ge 0$ on $0\le p\le \frac{1}{2}$, and the integral computation shows $\int_0^{\frac{1}{2}}\Delta(p)\,dp=0$. Hence, $\Delta(p)\equiv 0$ on $[0,\frac{1}{2}]$. Choose $p_0\in(0,\frac{1}{2})$. Equality throughout the chain implies
$$
a=\mu_{p_0}(\mathcal{F}^{\uparrow})=p_0,\qquad d=\mu_{1-p_0}(\mathcal{F}^{\downarrow})=\mu_{p_0}(\overline{\mathcal{F}^{\downarrow}})=p_0.
$$
By  Theorem \ref{le1}, there exist $i,j\in[n]$ such that
$$
\mathcal{F}^{\uparrow}=\{X\subseteq[n]:i\in X\},\qquad \overline{\mathcal{F}^{\downarrow}}=\{X\subseteq[n]:j\in X\}.
$$
 Since $\mathcal{F}\subseteq\mathcal{F}^{\uparrow}\cap\mathcal{F}^{\downarrow}$, we have
$$
\mathcal{F}\subseteq \{X\subseteq[n]:i\in X,\ j\notin X\}.
$$
If $i=j$, then the right-hand side is empty, contradicting equality. Hence, $i\ne j$. Let
$
\mathcal{E}_{i,j}=\{X\subseteq[n]:i\in X,\ j\notin X\}.
$
Its binomial norm is
$$
  \lVert\mathcal E_{i,j}\rVert_n=\sum_{X\in\mathcal{E}_{i,j}}\binom{n}{|X|}^{-1}
= \sum_{k=1}^{n-1}\frac{\binom{n-2}{k-1}}{\binom{n}{k}}
= \sum_{k=1}^{n-1}\frac{k(n-k)}{n(n-1)}
= \frac{n+1}{6}.
$$
Since every summand in the definition of the binomial norm is positive,
the inclusion $\mathcal F\subseteq\mathcal E_{i,j}$ together with
$
  \lVert\mathcal F\rVert_n
  =\lVert\mathcal E_{i,j}\rVert_n
$
forces $\mathcal F=\mathcal E_{i,j}$. This completes the proof.

\section*{Declaration of competing interest}
We declare that we have no conflict of interest to this work.

\section*{Data availability}
No data was used for the research described in the article.

\section*{Acknowledgments}
The authors thank Peter Frankl for his helpful comments on the
terminology and historical background of the binomial norm, and for
drawing their attention to his related work.
%The authors would like to express their sincere thanks to the referee for the valuable suggestions which greatly improved the presentation of the %manuscript.

\end{document}